\documentclass[12pt]{amsart}
\usepackage{amsmath}
\usepackage{mathrsfs}
\usepackage{graphicx}

\def\e{\epsilon}
\def\lf{\left}
\def\a{{\alpha}}

\def\bx{{\bar\xi}}

\def\p{\partial}
\newcommand\R{{\mathbb R}}

\def\be{\begin{equation}}
\def\ee{\end{equation}}
\def\bee{\begin{equation*}}
\def\eee{\end{equation*}}
\def\mbg{\mathbf{g}}
\def\bu{\mathbf{u}}

\def\lf{\left}
\def\ri{\right}
\def\a{{\alpha}}

\def\e{\epsilon}

\def\bx{{\bar\xi}}

\def\p{\partial}

\def\p{\partial}

\def\p{\partial}

\def\lf{\left}
\def\ri{\right}
\def\a{{\alpha}}

\def\e{\epsilon}

\def\bx{{\bar\xi}}

\def\p{\partial}

\def\p{\partial}

\def\p{\partial}

\def\ri{\right}

\def\la{\langle}
\def\ra{\rangle}

\def\sph{\mathbb{S}}
\def\vv<#1,#2>{\left\la#1,#2\right\ra}
\newtheorem{thm}{Theorem}[section]

\newtheorem{lem}{Lemma}[section]

\theoremstyle{definition}
\newtheorem{defn}{Definition}[section]
\theoremstyle{remark}
\newtheorem{rem}{Remark}
\numberwithin{equation}{section}
\def\tl{\tilde\lambda}
\def\tm{\tilde\mu}
\def\ts{\tilde\sigma}

\def\be{\begin{equation}}
\def\ee{\end{equation}}
\begin{document}
\title{Steklov eigenvalues on annulus}

\author{Xu-Qian Fan$^1$}
\address[Xu-Qian Fan]{Department of Mathematics, Jinan University, Guangzhou, 510632, China}
\email{txqfan@jnu.edu.cn}
\thanks{$^1$Research partially supported by the National Natural Science Foundation of China (10901072, 11101106)}
\author{Luen-Fai Tam$^2$}
\address[Luen-Fai Tam]{The Institute of Mathematical Sciences and Department of
 Mathematics, The Chinese University of Hong Kong,
Shatin, Hong Kong, China.} \email{lftam@math.cuhk.edu.hk}
\thanks{$^2$Research partially supported by Hong Kong  RGC General Research Fund
\#CUHK 403011}
\author{Chengjie Yu$^3$ }
\address[Chengjie Yu]{Department of Mathematics, Shantou University, Shantou, Guangdong, 515063, China}
\email{cjyu@stu.edu.cn}
\thanks{$^3$Research partially supported by GDNSF S2012010010038 and the National Natural Science Foundation of
China (11001161).}
\renewcommand{\subjclassname}{
  \textup{2010} Mathematics Subject Classification}

\date{}
\begin{abstract}
We obtain supremum of the $k$-th normalized Steklov eigenvalues of all rotational symmetric conformal metrics on $[0,T]\times \mathbb{S}^1$, $k>1$. The case $k=1$ for all conformal metrics on $[0,T]\times \mathbb{S}^1$ has been completely solved by Fraser and Schoen \cite{FraserSchoen-2011,FraserSchoen-2012,FraserSchoen-2013}. We  give geometric description in terms of minimal surfaces for metrics attaining the supremum. We also obtain some partial results on the comparison of the normalized Stekov eigenvalues of rotationally symmetric metrics and general conformal metrics on $[0,T]\times \mathbb{S}^1$. A counter example is constructed to show that for fixed $T$ the first normalized Steklov eigenvalue of  rotationally symmetric metric may not be larger.
\end{abstract}
\maketitle

\markboth{ }
{ }

\section{Introduction}
Let $(M,g)$ be a compact Riemannian manifold of dimension not less than $2$ with nonempty boundary $\partial M$ and $u$ be a smooth function on $\partial M$. We denote the harmonic extension of $u$ on $M$ as $\hat u $. Then, the Dirichlet-to-Neumann map $L_g$ sends $u$ to $\frac{\partial \hat u}{\partial n }$ where $n$ means the unit outward normal on $\partial M$. The eigenvalues of $L_g$ are called Steklov eigenvalues which were first introduced by Steklov \cite{Steklov} in 1902.  $L_g$ is a nonnegative self-adjoint first order elliptic pseudo-differential operator (see \cite{CEG2001}). The spectrum of $L_g$ is discrete and unbounded:
$$0=\sigma_0(g)<\sigma_1(g)\leq \sigma_2(g)\leq\cdots\leq\sigma_k(g)\leq \cdots.$$

The Dirichlet-to-Neumann map is important in Electrical Impedance Tomography which is closely related to an inverse problem raised by Calder\'on \cite{Calderon1980}. In the problem of Calder\'on, $u$ means the boundary voltage and $\frac{\partial \hat u}{\partial n}$ means the boundary current. The problem is to recover $g$ from the boundary measurements of the voltage and current. It was shown in \cite{LU2001,LTU2003} that this can be done for $\mbox{dim} M\geq 3$ and $g$ real analytic. When $M$ is a surface, It was also shown in \cite{LU2001,LTU2003} that we can only recover the conformal class of $g$.

There have been many works on estimating the Steklov eigenvalues, see \cite{CEG2001,Escobar2000,FraserSchoen-2011,FraserSchoen-2012,FraserSchoen-2013} and references therein. In this paper, we will only consider the Steklov eigenvalues on an annulus (Riemann surface with genus zero and two boundaries). When $M$ is a surface,
$$\tilde \sigma_k(g)=\sigma_k(g)L(\partial M)$$ is called the $k$-th
normalized Steklov eigenvalue where $L(\partial M)$ means the length of $\partial M$. In \cite{FraserSchoen-2011}, Fraser and Schoen computed the maximum the first normalized Steklov eigenvalue on the annulus among all rotationally symmetric metrics and found that the maximum is achieved by the critical catenoid which is a portion of the catenoid that meets the boundary of the ball orthogonally. In \cite{FraserSchoen-2013}, by using minimal surfaces as in \cite{LY1982}, Fraser and Schoen showed that the maximum of the first normalized Steklov eigenvalues for the annulus is achieved by the critical catenoid. For simply connected planar domain, it is a classical result by Weinstock \cite{W} which says that the maximum of the first normalized Steklov eigenvalue is achieved by the round disk in the Euclidean plane. This result was also extended to any Riemann surface with genus zero and one boundary in \cite{FraserSchoen-2011}.

In this paper, motivated by \cite{FraserSchoen-2011}, we first compute the supremum   of all the normalized Steklov eigenvalues among all rotationally symmetric metrics. Let $M_k$ be the supremum of the $k$-th normalized Steklov eigenvalue of   rotational conformal metrics $f^2(t)(dt^2+d\theta^2)$ on a cylinder $[0,T]\times \mathbb{S}^1$. We have the following:

\begin{thm}\label{thm-1-intro}
$$M_{2k-1}=\frac{4k\pi}{T_{2,0}(1)},
$$
and is achieved when and only when $f(1)=f(T)$ and $T=\frac2kT_{2,0}(1)$, where $T_{2,0}(1)$ is the unique positive root of $s=\coth s$. $M_2=4\pi$ and is achieved when  $f(1)=f(T)$ and $T=\infty$. For $k>1$
$$
M_{2k}=4k\pi\tanh\lf(\frac{kT_{k,1}(1)}{2}\ri)
$$
where $T_{k,1}(1)$ is the unique positive root   of $k\tanh(\frac{ks}{2})=\coth(\frac{s}{2})$. $M_{2k}$ is achieved when and only when $f(1)=f(T)$, $T=T_{k,1}(1)$.

\end{thm}

 We find that all the supremums are achieved by an embedded or immersed minimal surfaces meeting the boundary of the ball orthogonally except for the second normalized Steklov eigenvalues whose supremum can not be achieved, see section 3 for more details. Note that in \cite{GP2010}, Girouard and  Polterovich showed that the supremum of the second normalized Steklov eigenvalue for simply connected planar domain is $4\pi$ and cannot be achieved. Combining our computation and the result in \cite{GP2010}, it maybe natural to conjecture that the supremum of the second normalized Steklov eigenvalue can not be achieved on any surfaces.

In \cite{Escobar2000}, motivated by Cheng \cite{Cheng1975}, Escobar studied the comparison of first Steklov eigenvalue. In \cite{FraserSchoen-2011}, Fraser and Schoen also compared the first  normalized Steklov eigenvalue of a supper critical metric on the annulus with the first  normalized Steklov eigenvalue of the critical catenoid. Motivated by all these results, in the second part of this paper, we compare all the Steklov eigenvalues of a general metric and the rotationally symmetric metric on the annulus. It turns out that the comparison is true for a large class of metrics (See Theorem \ref{thm-eigen-comp-1}, Theorem \ref{thm-eigen-comp-2}), but is not true in general (See Theorem \ref{thm-example}).

\section{Upper estimates of normalized eigenvalues}

Let  $\Sigma=[0,T]\times \sph^1$ equipped  with the metric
\begin{equation}\label{metric-1}
g=f(t)^2(dt^2+d\theta^2).
\end{equation}
We want to compute the maximum of the nonzero normalized  eigenvalues $\tilde\sigma_k(g)=\sigma_k(g)L_g(\partial \Sigma)$ with $g$ in the form \eqref{metric-1} for $k> 0$ and for all $T>0$. Here $\sigma_k(g)$ is the $k$-th Stekolov eigenvalue.

Let $$\a=\frac{f(0)}{f(T)}, \text{\ and \ }\beta=\frac{4(f(0)f(T))^{-1}}{\lf((f(0) )^{-1}+(f(T))^{-1}\ri)^2}= \frac{4\a}{(1+\a)^2}.$$
It is well-know that $\tilde\sigma_k(g)$ depends only on $\a$ (or equivalently $\beta$) and $T$, see \cite{FraserSchoen-2013}. By symmetry we may assume that $\a\ge1$. Therefore we also denote, $\tilde\sigma_k(g)$ by $\tilde\sigma_k(\beta,T)$ if $g$ and $\beta$ are related as above. In this section, we want to compute

\begin{equation}\label{eq-max-1}
M_k=\sup_{T>0,\alpha\geq1}\tilde\sigma_k(\beta,T).
\end{equation}
For $k=1$, this has been obtained in \cite{FraserSchoen-2011}. In fact, sharp bound for $k=1$ for general conformal metrics on $\Sigma$ is also obtained in \cite{FraserSchoen-2013}.

In the remaining of this paper, we always assume that $\a\ge1$. Note that $\beta\le1$ and $\beta=1$ if and only if $\a=1$.

By \cite{FraserSchoen-2011}, all the nonzero normalized Steklov eigenvalues of $g$ are as follows:
\begin{equation}
\tilde\lambda_n(\beta,T)= \frac{4n\pi}{\beta}\left(\coth(nT)-\sqrt{\coth^2(nT)-\beta}\right),
\end{equation}
\begin{equation}
\tilde\mu_0(\beta,T)= \frac{8\pi}{T\beta}
\end{equation}
and
\begin{equation}
\tilde\mu_n(\beta,T)= \frac{4n\pi}{\beta}\left(\coth(nT)+\sqrt{\coth^2(nT)-\beta}\right),
\end{equation}
 for $n=1,2,\cdots$. So the question is to find out which of the $\tilde\lambda_m$ or $\tilde\mu_n$ gives $\tilde\sigma_k$ and to estimate its value.

\begin{lem}\label{lem-eigen}

\item[(i)] $\tilde\lambda_n<\tilde\lambda_{n+1}$,  $\tilde\mu_{n-1}<\tilde\mu_n$, and $\tilde\lambda_n<\tilde\mu_n$ for all $n\ge 1$. Moreover, each $\tilde\lambda_n$, and each $\tilde\mu_n$ has multiplicity two, for $n\ge1$.
\item[(ii)]\begin{equation*}
\frac{\partial\tilde\lambda_n}{\partial T}=\frac{n\tilde\lambda_n}{\sinh^2(nT)\sqrt{\coth^2(nT)-\beta}}
\end{equation*}
and
\begin{equation*}
\frac{\partial\tilde\mu_n}{\partial T}=-\frac{n\tilde\mu_n}{\sinh^2(nT)\sqrt{\coth^2(nT)-\beta}}.
\end{equation*}
for $n= 1,2\cdots$.
 $$
 \frac{\p \tm_0}{\p T}=-\frac{8\pi}{T^2\beta}.$$
 In particular, $\tilde\lambda_n(\beta,T)$ is increasing in $T$, and $\tilde\mu_n(\beta,T)$ is decreasing in $T$.
\item[(iii)]
\bee
 \tl_n(\beta,\infty):=\lim_{T\to\infty} \tl_n(\beta,T)=\frac{4n\pi}{1+\sqrt{1-\beta}}=
 \frac{2n\pi(1+\a)}{\a};
\eee

\bee
 \tm_n(\beta,\infty):=\lim_{T\to\infty}\tm_n(\beta,T)=\frac{4n\pi}{1-\sqrt{1-\beta}}=
2n\pi(1+\a).
\eee
Hence if $n\ge1$, then $\tm_n(\beta,\infty)=\a\tl_n(\beta,\infty)$
\bee
\tl_n(\beta,0):=\lim_{T\to0} \tl_n(\beta,T)=0;
\eee
\bee
\tm_n(\beta,0):=\lim_{T\to0} \tl_n(\beta,T)=\infty,
\eee
for $n=1,2,\cdots$.
 \item[(iv)]
\begin{equation*}
\begin{split}
\frac{\partial \tilde\lambda_n}{\partial\beta}=&
\frac{\tilde\lambda_n}{2\sqrt{\coth^2(nT)-\beta}\left(\coth(nT)+\sqrt{\coth^2(nT)-\beta}\right)}\\
=&\frac{2n\pi \tl_n} {\tm_n\beta\sqrt{\coth^2(nT)-\beta}}\\
=&\frac{2 \pi  \sinh^2(nT)  }{\tm_n \beta} \frac{\p \tl_n}{\p T},
\end{split}
\end{equation*}
and
\begin{equation*}
\begin{split}
\frac{\partial \tilde\mu_n}{\partial\beta}=&
-\frac{\tilde\mu_n}{2\sqrt{\coth^2(nT)-\beta}\left(\coth(nT)-\sqrt{\coth^2(nT)-\beta}\right)}\\
=&-\frac{2n\pi \tm_n} {\tl_n\beta\sqrt{\coth^2(nT)-\beta}}\\
=&\frac{2 \pi  \sinh^2(nT)  }{\tl_n^2\beta} \frac{\p \tm_n}{\p T}
\end{split}
\end{equation*}
for $n\ge1$.
$$
\frac{\p \mu_0}{\p\beta}=-\frac{8\pi}{T\beta^2}.
$$
In particular, $\tilde\lambda_n(\beta,T)$ is increasing in $\beta$, and $\tilde\mu_n(\beta,T)$ is decreasing in $\beta$.
\item[(v)] $\tilde\lambda_n(1,T)=4n\pi\tanh(\frac{nT}{2})$ and $\tilde\lambda_n(0,T):=\lim_{\beta\to0}\tl_n(\beta,T) =2n\pi\tanh(nT)$. $\tilde\mu_n(1,T)=4n\pi\coth(\frac{nT}{2})$, $\tilde\mu_n(0,T):=\lim_{\beta\to0}\tm_n(\beta,T) =\infty$.

\end{lem}
\begin{proof}
(i) has been observed in \cite{FraserSchoen-2011}. In fact, except for the case $\tm_0<\tm_1$, other inequalities are obvious. Now
 $h(x)=x\left(\coth(x)+\sqrt{\coth^2(x)-\beta}\right)$
is increasing and
\bee
h(x)\to 2 \ \mbox{as}\ x\to 0^+.
\eee
From these, it is easy to see that $\tilde\mu_1>\tilde\mu_0$.

(ii)--(v) follow from direct computations,
\end{proof}

\begin{lem}\label{lem-intersection-1} Let $k\ge1$, $l\ge0$. For fix $\beta$, $\tl_k(\beta,T)=\tm_l(\beta,T)$ for some $T>0$ if and only if $\a<\frac kl$. Moreover, $T$ is unique if it exists. By convention if $l=0$, then $\frac kl=\infty$.
\end{lem}
\begin{proof} First suppose   $l=0$, then by Lemma \ref{lem-eigen}, $\tl_k(\beta,T)-\tm_0(\beta,T)$ is increasing in $T$, $\tl_k(\beta,0)-\tm_0(\beta,0)= -\infty$,
$\tl_k(\beta,\infty)-\tm_0(\beta,\infty)= 2k\pi(1+\a)\a^{-1}$. Hence there is a unique $T>0$ such that $\tl_k(\beta,T)-\tm_0(\beta,T)=0$. This proves that case that $k\ge1$, $l=0$.

Suppose $l\ge1$. Again $\tl_k(\beta,T)-\tm_0(\beta,T)$ is increasing in $T$, $\tl_k(\beta,0)-\tm_l(\beta,0)= -\infty$. Hence there is a unique $T>0$ such that  $\tl_k(\beta,0)-\tm_l(\beta,0)= 0$ if any only if $\tl_k(\beta,\infty)-\tm_0(\beta,\infty)>0$. By Lemma \ref{lem-eigen} again:
$$
 \tl_k(\beta,\infty)-\tm_l(\beta,\infty)=2 \pi(1+\a)\lf( \frac k\a-l \ri)
 $$
 which is positive if and only if $\a<\frac kl$.
\end{proof}

\begin{defn}
  For $\a<\frac kl$ let $T_{k,l}(\beta)$ be the unique positive number such that
\begin{equation}
\tilde \lambda_k(\beta,  T_{k,l}(\beta))=\tilde \mu_l(\beta,  T_{k,l}(\beta)).
\end{equation}
\end{defn}
Since $\a\ge1$, we must have $k>l$. Note that $1\le \a<\frac kl$ if and only if $a_{kl}<\beta\le 1$ where $a_{kl}=\frac{4kl}{(k+l)^2}$.
\begin{lem}\label{lem-intersection-2}
$T_{k,l}(\beta)$ is decreasing in $k$ and increasing in $l$ whenever defined.
\end{lem}
\begin{proof} The lemma follows from Lemma \ref{lem-eigen}(i).
\end{proof}

Let $k\ge1$ and $\a\ge1$. There is a unique integer $s\ge 0$ be the largest integer such that
$$
\frac{k-s}{s}>\a.
$$
If $s=0$, then $(k-s)/s$ is $\infty$ be convention. In this case, $\a\ge k-1$. Then the following is true,
$$
\frac{k-j}{j-1}>\frac{k-j}{j}\ge\frac{k-s}{s}>\a
$$
for $1\le j\le s$. Again $j=1$ means that $(k-j)/(j-1)=\infty$. Hence for $1\le j<s$, $T_{k-j,j-1}(\beta)<T_{k-j-1,j}(\beta)$  are defined. In the following, denote $T_{m,n}(\beta)$ simply by $T_{m,n}$.  Then for $s\ge1$,
$$
(0,\infty)=(0,T_{k-1,0}]\bigcup\lf(\cup_{j=1}^{s-1}[T_{k-j,j-1},T_{k-j-1,j})\ri)
\bigcup[T_{k-s,s-1}, \infty)
$$
If $s=0$, we simply have:

$$
(0,\infty)=(0,T_{k-1,0}]\cup[T_{k-1,0},\infty)
$$

\begin{lem}\label{lem-sigma-odd} With the above notations and assumptions for $k\ge 1$:

 \begin{itemize}
   \item [(i)] if $s=1$, then
\bee
\ts_{2k-1}(\beta,T)\le\left\{
                        \begin{array}{ll}
                          \tl_k(\beta,T_{k,0}), & \hbox{for $T\in(0, T_{k-1,0})$;} \\
                          \tl_{k-j}(\beta,T_{k-j,j}), & \hbox{for $T\in[T_{k-j,j-1},T_{k-j-1,j})$,}
                          \\
                          &\hbox{\ \ and $1\le j\le s-1$;} \\
                          \max\{\tl_{k-s}(\beta,T_{k-s,s}),\tl_{k-s-1}(\beta,\infty)\} & \hbox{for $T\in [T_{k-s,s-1},\infty)$;}
                        \end{array}
                      \right.
\eee
   \item [(ii)] if $s=0$, then
   \bee
\ts_{2k-1}(\beta,T)\le\left\{
                        \begin{array}{ll}
                          \tl_k(\beta,T_{k,0}), & \hbox{for $T\in(0, T_{k-1,0})$;} \\

                          \tl_{k-1}(\beta,\infty) & \hbox{for $T\in [T_{k-1,0},\infty)$;}
                        \end{array}
                      \right.
\eee
 \end{itemize}

\end{lem}
\begin{proof} (i) If $s\ge1$, $1\le j\le s-1$ and  if $T\in[T_{k-j,j-1},T_{k-j-1,j})$, then the first $2k-3$ normalized eigenvalues are given by $\tl_{p}(\beta,T)$, $1\le p\le k-j-1$ each with multiplicity two,
 $\tm_{q}(\beta,T)$, $1\le q\le j-1$, each with multiplicity two, and $\tm_0(\beta,T)$. So $\ts_{2k-1}(\beta,T)$ is equal to $\tl_{k-j}(\beta,T)$ or $\tm_{j}(\beta,T)$. The first case occurs only when $T\le T_{k-j,j}$ and the second case occurs only when $T>T_{k-j,j}$. Hence by Lemma \ref{lem-eigen}, we have $\ts_{2k-1}(\beta,T)\le \tl_{k-j}(\beta,T_{k-j,j})=\tm_{j}(\beta,T_{k-j,j})$.

  One can use similar method to prove that:
  $\ts_{2k-1}(\beta,T)\le \tl_{k}(\beta,T_{k,0})$  if $T\in(0,T_{k-1,0})$, $\ts_{2k-1}(\beta,T)\le\tl_{k-s}(\beta,T_{k-s,s})$ if   $T\in[T_{k-s,s-1},T_{k-s,s})$.
  Suppose $(k-s-1)/s>\a$, then for $T\in[T_{k-s,s},T_{k-s-1,s})$, $\ts_{2k-1}=\tm_s(\beta,T )$ which is less than $\tm_s(\beta,T_{k-s,s} )=\tl_{k-s}(\beta,T_{k-s,s})$.
  $\ts_{2k-1}(\beta,T)\le \tl_{k-s-1}(\beta,\infty)$ if  $T\in [T_{k-s-1,s},\infty)$ using the fact that $\tm_{s+1}(\beta,T)>\tl_{k-s-1}$ for all $T$. Suppose   $(k-s-1)/s\le \a$, then $\ts_{2k-1}=\tm_s(\beta,T )$ which is less than $\tm_s(\beta,T_{k-s,s} )=\tl_{k-s}(\beta,T_{k-s,s})$ for $T\in[T_{k-s,s},\infty)$.

  (ii) can be proved similarly.

  This completes the proof of the lemma.

\end{proof}
As before, for  $k\ge1$ and $\a\ge1$, there is a unique integer $s\ge 0$ be the largest integer such that
$$
\frac{k-s}{s}>\a.
$$
\begin{lem}\label{lem-sigma-even} With the above notations and assumptions for $k\ge 1$:

 \begin{itemize}
   \item [(i)] if $\ts_{2}(\beta,T)=\tl_1(\beta,T)$.

   \item[(ii)] if $k\ge 2$ and $s\ge1$
\bee
\ts_{2k }(\beta,T)\le\left\{
                        \begin{array}{ll}
                          \tl_k(\beta,T_{k,0}), & \hbox{for $T\in(0, T_{k,0})$;} \\
                          \tl_{k-j}(\beta,T_{k-j,j+1}), & \hbox{for $T\in[T_{k-j,j},T_{k-j-1,j+1})$,}
                          \\
                          &\hbox{\ \ and $0\le j\le s-1$;} \\
                           \tl_{k-s}(\beta,T_{k-s,s+1}) & \hbox{for $T\in [T_{k-s,s},\infty)$, and if $(k-s)/(s+1)>\a$;}\\
                            \tl_{k-s}(\beta,\infty) & \hbox{for $T\in [T_{k-s,s},\infty)$, and if $(k-s)/(s+1)\le \a$;}
                        \end{array}
                      \right.
\eee
   \item [(iii)] if $s=0$, then
   \bee
\ts_{2k}(\beta,T)\le\left\{
                        \begin{array}{ll}
                          \tl_k(\beta,T_{k,0}), & \hbox{for $T\in(0, T_{k,0})$;} \\
\tl_{k}(\beta,T_{k,1}) & \hbox{for $T\in [T_{k-1,0},\infty)$ if $\a<k$;}
\\
\tl_{k}(\beta,\infty) & \hbox{for $T\in [T_{k-1,0},\infty)$ if $\a\ge k$;}
                        \end{array}
                      \right.
\eee
 \end{itemize}

\end{lem}
\begin{proof}
(i) is obvious.

(ii) Suppose $0\le j\le s-1$ so that $T_{k-j,j}<T_{k-j-1,j+1}$ are defined. Suppose $ T\in[T_{k-j,j},T_{k-j-1,j+1})$. Then the first $2k-1$ normalized eigenvalues are given by
$\tl_p(\beta,T)$ for $0\le p\le k-j-1$  each with multiplicity two, $\tl_q(\beta,T)$ for $1\le q\le j-1$ each with multiplicity two and $\tm_0(\beta,T)$. Hence $\ts_{2k}(\beta,T)$ is equal to $\tl_{k-j}(\beta,T)$ is $T\in [T_{k-j,j},T_{k-j,j+1})$ and is equal to $\tm_{j+1}(\beta,T)$ if
$T\in [T_{k-j,j+1},T_{k-j-1,j+1})$. In any case, $ \ts_{2k}(\beta,T)\le \tl_{k-j}(\beta,T_{k-j ,j+1})$ by Lemma \ref{lem-eigen}.

Let $T\in[T_{k-s,s},\infty)$. Suppose $(k-s)/(s+1)>\a$. Then similarly, we can prove that
$\ts_{2k}$ is equal to $\tl_{k-s}(\beta,T)$ if $T\in [T_{k-s,s},T_{k-s,s+1})$ and is equal to
$\tm_{s+1}(\beta,T)$ if $T\in [T_{k-s,s+1},\infty)$ . In any case, $\ts_{2k}\le\tl_{k-s}(\beta,T_{k-s,s+1})$.
Suppose $(k-s)/(s+1)\le\a$. Then  $\ts_{2k}$ is equal to $\tl_{k-s}(\beta,T)$ which is less than or equal to $\tl_{k-s}(\beta,\infty)$.

The proof of (ii) is similar.
\end{proof}

By Lemma \ref{lem-sigma-odd}, in order to estimate $M_{2k-1}$, it is sufficient to estimates $\tl_k(\beta,T_{k,0}(\beta))$, $\tl_{k-j}(\beta,T_{k-j,j}(\beta))$ and $\tl_{k-s-1}(\beta,\infty)$.
\begin{lem}\label{lem-upperbound-odd-1}
We have the following:
\begin{itemize}
  \item [(i)] $\tl_k(\beta,T_{k,l}(\beta))< \tl_k(1,T_{k,l}(1))$ for all $k>l\ge0$.
  \item [(ii)] $\tl_k(1,T_{k,l}(1))<\tl_{k+c,l-c}(1)$ for $k>l\ge c>0$.
  \item [(iii)] $kT_{k,0}(1)=2T_{2,0}(1)$ for $k\ge2$, where
  $T_{2,0}(1)=\coth T_{2,0}(1)$. Hence $T_{2,0}(1)\simeq 1.2$.
\end{itemize}
\end{lem}
\begin{proof} (i)  First assume that $l\ge1$. By definition
$$
\tl_k(\beta,T_{k,l}(\beta))=\tm_l (\beta,T_{k,l}(\beta)).
$$
In the following, denote $T_{k,l}(\beta)$ simply by $T(\beta)$.
By Lemma \ref{lem-eigen},
\bee
\begin{split}
\frac{d}{d\beta}(\tl_k(\beta,T(\beta)))=&
\frac{\p\tl_k}{\p\beta}+\frac{\p\tl_k}{\p T}\frac{dT}{d\beta}\\
=& \lf(\frac{2 \pi  \sinh^2(kT)  }{\tm_k \beta}  +\frac{dT}{d\beta}\ri) \frac{\p \tl_n}{\p T}\\
=&\lf(\frac{ \tl_k  \sinh^2(kT)   }{8\pi k^2    }  +\frac{dT}{d\beta}\ri) \frac{\p \tl_k}{\p T}
\end{split}
\eee
\bee
\begin{split}
\frac{d}{d\beta}(\tm_l(\beta,T(\beta)))=&
\frac{\p\tm_l}{\p\beta}+\frac{\p\tm_l }{\p T}\frac{dT}{d\beta}\\
=& \lf(\frac{\tm_l  \sinh^2(lT)  }{8\pi l^2}+\frac{dT}{d\beta}\ri) \frac{\p \tm_l}{\p T}
\end{split}
\eee
where we have used the fact that $\tl_n\tm_n=(16n^2\pi^2)/\beta$. Suppose
$\frac{d}{d\beta}(\tl_k(\beta,T(\beta)))\le0$, then
$$
\frac{dT}{d\beta}\le-\frac{ \tl_k  \sinh^2(kT)   }{8\pi k^2    }.
$$
because $\frac{\p \tl_k}{\p T}>0$. Hence
\bee
\begin{split}
\frac{d}{d\beta}(\tm_l(\beta,T(\beta)))\ge &
  \lf(\frac{\tm_l  \sinh^2(lT)  }{8\pi l^2}-\frac{ \tl_k  \sinh^2(kT)   }{8\pi k^2    }\ri) \frac{\p \tm_l}{\p T}\\
  >0
\end{split}
\eee
because $ \frac{\p \tm_l}{\p T}<0$, $\tl_k=\tm_l$,  and $k>l$ which implies $\sinh^2 (lT)/l^2<\sinh^2l(kT)/k^2$. However,
$\tm_l(\beta,T(\beta))=\tl_k(\beta,T(\beta))$, this is a contradiction. Hence (i) is true if $l\ge1$. The proof that $l=0$ is similar.

(ii) follows from Lemma \ref{lem-tanh-coth} in the following.

(iii) $T_{k,0}(1)$ is the unique positive $T$ such that
$$
4k\pi \tanh \lf(\frac{k}2T_{k,0}(1)\ri)=\frac{8\pi}{T_{k,0}(1)}.
$$
So
$$
4(k-1)\pi\tanh \lf(\frac{k-1}2\frac{k}{k-1}T_{k,0}(1)\ri)=\frac{8\pi}{\frac{k}{k-1}T_{k,0}(1)}.
$$
From this, we conclude that
$$
T_{k-1,0}=\frac{k}{k-1}T_{k,0}(1).
$$
So (iii) is true.
\end{proof}
\begin{lem}\label{lem-f}
Then function $f_\beta(t)=t^{-2}(\sinh^2 t\sqrt{\coth^2t-\beta}-t)$ is increasing for $t>0$.
\end{lem}
\begin{proof}
Note that
\begin{equation*}
\begin{split}
f'_\beta(t)=2t^{-3}\sinh t(t\cosh t-\sinh t)\sqrt{\coth^2 t-\beta}-t^{-2}\left(\coth^2 t-\beta\right)^{-\frac{1}{2}}\coth t+t^{-2}
\end{split}
\end{equation*}
is decreasing on $\beta$ since $t\cosh t-\sin t\geq 0$. So, we only need to check that
$f'_1(t)\geq 0$. Note that
$f'_1(t)=-2t^{-3}\sinh t+t^{-2}\cosh t+t^{-2}$. So we only need to check that
$$h(t)=t+t\cosh t-2\sinh t\geq 0$$
for $t\geq 0$. For this we only need to check
$$h'(t)=1+t\sinh t-\cosh t\geq 0$$
for $t\geq 0$. This is clear since
$$h''(t)=t\cosh t\geq 0$$
for $t\geq 0$.
\end{proof}
\begin{lem}\label{lem-tanh-coth}
Let $x(a,b)$ be the solution of
\begin{equation}\label{eqn-a-b}
a\left(\coth(ax)-\sqrt{\coth^2(ax)-\beta}\right)=b\left(\coth(bx)+\sqrt{\coth^2(bx)-\beta}\right)
\end{equation}
 for $a>b\geq 0$. Let
$$u(a,b)=a\left(\coth(ax(a,b))-\sqrt{\coth^2(ax(a,b))-\beta}\right).$$
 Then, $u(a,b)<u(a+c,b-c)$ if $a>b\geq c>0$.
\end{lem}
\begin{proof}

By taking derivative with respect to $a$ and $b$ on \eqref{eqn-a-b}, we know that
\begin{equation}
\begin{split}
&\frac{\partial x}{\partial a}\\
=&-\left(\frac{a^2\left((\coth(ax)-\sqrt{\coth^2(ax)-\beta}\right)}{\sinh^2(ax)\sqrt{\coth^2(ax)-\beta}}+\frac{b^2\left(\coth(bx)+\sqrt{\coth^2(bx)-\beta}\right)}{\sinh^2(bx)\sqrt{\coth^2(bx)-\beta}}\right)^{-1}\\
&\times\left(\coth(ax)-\sqrt{\coth^2(ax)-\beta}\right)\left(1+\frac{ax}{\sinh^2(ax)\sqrt{\coth^2(ax)-\beta}}\right)
\end{split}
\end{equation}
\begin{equation}
\begin{split}
&\frac{\partial x}{\partial b}\\
=&\left(\frac{a^2\left((\coth(ax)-\sqrt{\coth^2(ax)-\beta}\right)}{\sinh^2(ax)\sqrt{\coth^2(ax)-\beta}}+\frac{b^2\left(\coth(bx)+\sqrt{\coth^2(bx)-\beta}\right)}{\sinh^2(bx)\sqrt{\coth^2(bx)-\beta}}\right)^{-1}\\
&\times\left(\coth(bx)+\sqrt{\coth^2(bx)-\beta}\right)\left(1-\frac{bx}{\sinh^2(bx)\sqrt{\coth^2(bx)-\beta}}\right)
\end{split}
\end{equation}
Note that $u(a,b)=b\coth(bx(a,b))$, so
\begin{equation}
\frac{\partial u}{\partial a}=-\frac{b^2\left(\coth(bx)+\sqrt{\coth^2(bx)-\beta}\right)}{\sinh^2(bx)\sqrt{\coth^2(bx)-\beta}}\frac{\partial x}{\partial a}.
\end{equation}
Similarly
\begin{equation}
\frac{\partial u}{\partial b}=\frac{a^2\left(\coth(ax)-\sqrt{\coth^2(ax)-\beta}\right)}{\sinh^2(ax)\sqrt{\coth^2(ax)-\beta}}\frac{\partial x}{\partial b}.
\end{equation}
In particular, $\frac{\p u}{\p a}>0$ and $\frac{\p u}{\p b}>0$.
Hence
\begin{equation}
\begin{split}
&\frac{\partial u}{\partial a}\bigg/\frac{\partial u}{\partial b}\\
=&\frac{a^{-2}\left(\sinh^2(ax)\sqrt{\coth^2(ax)-\beta}+ax\right)}{b^{-2}\left(\sinh^2(bx)\sqrt{\coth^2(bx)-\beta}-bx\right)}\\
>&\frac{a^{-2}\left(\sinh^2(ax)\sqrt{\coth^2(ax)-\beta}-ax\right)}{b^{-2}\left(\sinh^2(bx)\sqrt{\coth^2(bx)-\beta}-bx\right)}\\
>&1
\end{split}
\end{equation}
by Lemma \ref{lem-f} and that $a>b$. Suppose $a>b\ge t>0$.
  Let $f(t)=u(a+t,b-t)$, then
$$f'=\frac{\partial u}{\partial a}-\frac{\partial u}{\partial b}>0.$$
From this we get the conclusion.
\end{proof}
\begin{thm}\label{thm-odd-upperbound}
$M_{2k-1}=\frac{4k\pi}{T_{2,0}(1)}$ is achieved only when $\alpha=1$ and $T=T_{k,0}(1)=\frac{2T_{2,0}(1)}{k}$.
\end{thm}
\begin{proof} Let $\a\ge1$, $k\ge1$. Let $s$ be the smallest integer such that
$$
\a<\frac{k-s}s.
$$
Suppose $s\ge1$, then by Lemmas \ref{lem-sigma-odd},   \ref{lem-upperbound-odd-1} and \ref{lem-eigen}
$$
M_{2k-1}\le \max\lf\{\tl_k(1,T_{k,0}(1)),\frac{2(k-s-1)\pi(1+\a)}\a \ri\}.
$$
By Lemmas \ref{lem-eigen} and \ref{lem-tanh-coth},
$$
\tl_k(1,T_{k,0}(1)=\tm_0(1,T_{k,0}(1)=\frac{8\pi}{T_{k,0}(1)}=\frac{4k\pi}{T_{2,0}(1)}.
$$
On the other hand, since $\a\ge \frac{k-s-1}{s+1}$ and note that $k-s-1\ge0$,
\bee
\begin{split}
\frac{ (k-s-1) (1+\a)}\a\le & (s+1)+k-s-1\\
<& \frac{2k}{T_{2,0}(1)}
\end{split}
\eee
because $T_{2,0}(1)\simeq 1.2.$

If $s=0$, then by Lemmas \ref{lem-eigen} and \ref{lem-tanh-coth},
$$
M_{2k-1}\le \max\lf\{\tl_k(1,T_{k,0}(1)), \frac{2(k-1)\pi(1+\a)}\a\ri\}.
$$
\bee
\begin{split}
\frac{ (k-1) (1+\a)}\a\le & 1+k-1\\
<& \frac{2k}{T_{2,0}(1)}
\end{split}
\eee
because $\a\ge k-1$. Hence we conclude that
$$
M_{2k-1}\le \frac{4k\pi}{T_{2,0}(1)}.
$$
If $\a=1$, $T=T_{k,0}(1)$, then $\ts_{2k-1}(1,T_{k,0})=\tl_k(1,T_{k,0})=\frac{4k\pi}{T_{2,0}(1)}$. So
 $M_{2k-1}=\frac{4k\pi}{T_{2,0}(1)}$. On the other hand  if $\a>1$ or $T\neq T_{k,0}$, from the proof, we can see that $\ts_{2k-1}(\beta,T )<M_{2k-1}$.
\end{proof}
\begin{thm}\label{thm-even-upperbound}
 \begin{itemize}
                                       \item [(i)] $M_2=4\pi$ and is achieved only when $\a=1$, $T=\infty$
                                       \item [(ii)] For $k\ge2$,
                                       $$M_{2k}=\tl_k(1,T_{k,1}(1))=4k\pi\tanh\lf(\frac{kT_{k,1}(1)}2\ri)=4\pi\coth\lf(\frac{kT_{k,1}(1)}2\ri).$$
                                     It is achieved only when $\a=1$, $T=T_{k,1}(1)$.\end{itemize}
\end{thm}
\begin{proof}
(i) By Lemma \ref{lem-sigma-even}, $\sigma_2(\beta,T)=\tl_1(\beta,T)< \tl_1(1,\infty)=4\pi$.

(ii) For $k\ge 2$, and $s\ge 0$ be the largest integer so that $(k-s)>\a s$. Suppose $s\ge1$, by Lemmas \ref{lem-sigma-even}, \ref{lem-tanh-coth}, we have
$$
\ts_{2k}(\beta,T)\le \max\{\tl_k(\beta ,T_{k,0}(\beta)), \tl_{k}(1,T_{k,1}(1))\},
$$
if $(k-s)/(s+1)>\a$
and
$$
\ts_{2k}(\beta,T)\le \max\{\tl_k(\beta ,T_{k,0}(\beta)), \tl_{k}(1,T_{k,1}(1)),\tl_{k-s}(\beta,\infty)\},
$$
if   $(k-s)/(s+1)\le\a$.
By Lemma \ref{lem-tanh-coth}, $\tl_k(\beta,T_{k,0}(\beta))\le \tl_k(1,T_{k,0}(1))< \tl_k(1,T_{k,1}(1))$ by Lemmas \ref{lem-eigen} and \ref{lem-intersection-2}.

On the other hand, if $\a\le (k-s)/(s+1)$, then
\bee
\begin{split}
\tl_{k-s}(\beta,\infty)=& 2(k-s)\pi\lf(1+\frac1\a\ri)\\
\le&2(k-s)\pi+2(s+1)\pi\\
<&\frac{8k\pi}{2T_{2,0}(1)}\\
=&\tl_{k}(1,T_{k,0}(1))\\
<&\tl_{k}(1,T_{k,1}(1))
\end{split}
\eee
because $T_{2,0}(1)\simeq 1.2$ and $k\ge2$. In any case, $M_{2k}\le \tl_k(1,T_{k,1}(1))$ which is equal to $\ts_{2k}(1,T_{k,1}(1))$. Hence
$M_{2k}\le \tl_k(1,T_{k,1}(1))$

If $s=0$, we can prove similarly that $M_{2k}\le \tl_k(1,T_{k,1}(1))$.

From the proof above, one can conclude that $M_{2k}$ is achieved only when
$\a=1$ and $T=T_{k,1}(1)$.
\end{proof}
\section{Geometric pictures for metrics attaining the maximal values}
In \cite{FraserSchoen-2011}, Fraser and Schoen used the first eigenfunctions to embed the cylinder with maximal normalized first Steklov eigenvalue among all rotationally symmetric metrics into Euclidean space and discovered that it is a portion of the catenoid which solves the free boundary value problem in a ball. In \cite{FraserSchoen-2011}, Fraser and Schoen called this the critical catenoid. In this section, we will show that similar conclusions are true for the other Steklov eigenvalues except for the second Steklov eigenvalue.

It was obtained in \cite{FraserSchoen-2011} that all the normalized Steklov eigenvalues of  the rotational metric \eqref{metric-1} on the cylinder $\Sigma=[0,T]\times \sph^1$ are
$\tilde\lambda_n(\beta,T)$ and $\tilde\mu_n(\beta,T)$ for $n=0,1,2,\cdots$ as in the last section with $\tilde\lambda_0=0$. Moreover, the eigen-space of $\tilde\lambda_n$ is generated by
\begin{equation}
x_n(t,\theta)=\cosh(n(t-\tau_n))\cos(n\theta)
\end{equation}
and
\begin{equation}
y_n(t,\theta)=\cosh(n(t-\tau_n))\sin(n\theta).
\end{equation}
where $\tau_n$ is the unique positive number satisfying
\begin{equation}
\tanh(n\tau_n)=\frac{1+\alpha}{2}\left(\coth(nT)-\sqrt{\coth^2(nT)-\beta}\right)
\end{equation}
with $n=1,2,\cdots$. In particular, when $\alpha=1$,
\begin{equation}
\tau_n=\frac{T}{2}.
\end{equation}
The eigen-space of $\tilde \mu_0$ is generated by
\begin{equation}
z_0(t,\theta)=t-\frac{T}{1+\alpha}
\end{equation}
and the eigen-space of $\tilde\mu_n$ is generated by
\begin{equation}
z_n(t,\theta)=\sinh(n(t-\xi_n))\cos(n\theta)
\end{equation}
and
\begin{equation}
w_n(t,\theta)=\sinh(n(t-\xi_n))\sin(n\theta).
\end{equation}
where $\xi_n$ is the unique positive number such that
\begin{equation}
\coth (n\xi_n)=\frac{1+\alpha}{2}\left(\coth(nT)+\sqrt{\coth^2(nT)-\beta}\right)
\end{equation}
with $n=1,2,\cdots$. In particular, when $\alpha=1$,
\begin{equation}
\xi_n=\frac{T}{2}.
\end{equation}

Recall that the catenoid is an embedded minimal surface in $\R^3$ parameterized by
\begin{equation}
\left\{\begin{array}{l}
x(t,\theta)=t\\
y(t,\theta)=\cosh t\cos\theta\\
z(t,\theta)=\cosh t\sin\theta.
\end{array}\right.
\end{equation}
As defined in \cite{FraserSchoen-2011}, the critical catenoid is a portion of the catenoid with $t\in [-T_{2,0}(1),T_{2,0}(1)]$ which is characterized by that the boundary of the critical catenoid meets the boundary of the ball orthogonally.

Similarly, note that the following immersed catenoid
\begin{equation}
\left\{\begin{array}{l}
x(t,\theta)=nt\\
y(t,\theta)=\cosh (nt)\cos (n\theta)\\
z(t,\theta)=\cosh (nt)\sin (n\theta).
\end{array}\right.
\end{equation}
is also a minimal surface in $\R^3$. In fact, it is a $n$ to $1$ cover of the catenoid. We call the surface an $n$-catenoid. Similarly as in \cite{FraserSchoen-2011}, we called the portion of the $n$-catenoid with $t\in [-T_{2,0}(1)/n,T_{2,0}(1)/n]$ a critical $n$-catenoid which is also characterized by that the boundary of the critical $n$-catenoid meets the boundary of the ball orthogonally. Similarly as in \cite{FraserSchoen-2011}, we have the following conclusion by the computations in the last section.
\begin{thm}\label{thm-odd-eigenvalue}
The maximum of the  normalized $2n-1^{th}$ Steklov eigenvalue among all rotationally symmetric metrics of the form \eqref{metric-1} on the cylinder is achieved by
the $n$-critical catenoid immersed in $\R^3$.
\end{thm}
\begin{proof}
By Theorem \ref{thm-odd-upperbound}, the maximum of the normalized $2n-1^{th}$ Steklov eigenvalue $\tilde\sigma_{2n-1}(\beta,T)$ among all rotationally symmetric metrics of the form \eqref{metric-1} on the cylinder is achieved only when $\alpha=1$ and $T=T_{n,0}(1)=\frac{2T_{2,0}(1)}{n}$. In this case, the eigen-space of $\tilde\sigma_{2n-1}$ is generated by
$$z_0(t,\theta)=t-\frac{T_{n,0}(1)}{2}=t-\frac{T_{2,0}(1)}{n},$$
$$x_n(t,\theta)=\cosh\left(n(t-T_{2,0}(1)/n)\ri)\cos(n\theta)$$
and
$$y_n(t,\theta)=\cosh\lf(n(t-T_{2,0}(1)/n)\ri)\sin(n\theta)$$
with $t\in[0,2T_{2,0}(1)/n]$.

From this, we can use the eigen-functions of $\tilde\sigma_{2n-1}$ to immerse the cylinder into $\R^3$ as the follows:
\begin{equation}
\left\{\begin{array}{l}
x(t,\theta)=n\lf(t-\frac{T_{2,0}(1)}{n}\ri)\\
y(t,\theta)=\cosh \lf(n\lf(t-\frac{T_{2,0}(1)}{n}\ri)\ri)\cos (n\theta)\\
z(t,\theta)=\cosh \lf(n\lf(t-\frac{T_{2,0}(1)}{n}\ri)\ri)\sin (n\theta).
\end{array}\right.
\end{equation}
This gives us the conclusion.
\end{proof}
Recall that in \cite{FraserSchoen-2013}, the critical M\"obius band is a portion of the immersed M\"obius band in $\R^4$ parameterized by
\begin{equation*}
X(t,\theta)=(2\sinh t\cos\theta,2\sinh t\sin\theta,\cosh(2t)\cos(2\theta),\cosh(2t)\sin(2\theta))^T
\end{equation*}
with $t\in [-T_{2,1}(1)/2,T_{2,1}(1)/2]$ which is also characterized by that the boundary of the critical M\"obius band meets the boundary of the ball orthogonally.

Then, similarly as in Theorem \ref{thm-odd-eigenvalue} with suitable eigen-functions, we have the following conclusion.
\begin{thm}\label{thm-4-eigenvalue}
The maximum of the normalized $4^{th}$ Steklov eigenvalue among all rotational symmetric metrics on the cylinder is achieved by the critical Mobius band whose boundary meets the boundary of the ball orthogonally.
\end{thm}
For the other Steklov eigenvalues, similarly as in Theorem \ref{thm-odd-eigenvalue} by using suitable eigen-functions to embed or immerse the cylinder into $\R^4$, we have the following conclusions.
\begin{thm}\label{thm-4n-eigenvalue}
The maximum of the normalized $4n^{th}$ Steklov eigenvalue among all rotational symmetric metrics on the cylinder is achieved by the critical $n$-Mobius band immersed in $\R^4$ with boundary meets the boundary of the ball orthogonally. The immersed critical $n$-Mobius band is parameterized by
\begin{equation*}
X(t,\theta)=(2n\sinh t\cos\theta,2n\sinh t\sin\theta,\cosh(2nt)\cos(2n\theta),\cosh(2nt)\sin(2n\theta))^T
\end{equation*}
with $t\in [-T_{2n,1}(1)/2,T_{2n,1}(1)/2]$.
\end{thm}
\begin{thm}\label{thm-4n+2-eigenvalue}
The maximum of the normalized $4n+2^{th}$ Steklov eigenvalue with $n\geq 1$ among all rotational symmetric metrics on the cylinder is achieved by an embedded minimal surface in $\R^4$ with boundary meets the boundary of the ball orthogonally. The embedded minimal surface is parameterized by
\begin{equation*}
\begin{split}
&X(t,\theta)\\
=&((2n+1)\sinh t\cos\theta,(2n+1)\sinh t\sin\theta,\\
&\cosh((2n+1)t)\cos((2n+1)\theta),\cosh((2n+1)t)\sin((2n+1)\theta))^T
\end{split}
\end{equation*}
with $t\in [-T_{2n+1,1}(1)/2,T_{2n+1,1}(1)/2]$.
\end{thm}

\section{  Steklov eigenvalues of non-rotational symmetric metrics}
In this section, we will compare the Steklov eigenvalues of
a general conformal flat metric
\begin{equation}\label{metric-2}
\tilde g=f^2(t,\theta)(dt^2+d\theta^2)
\end{equation}
with the rotationally symmetric metrics as in \eqref{metric-1} on $\Sigma=[0,T]\times \sph^1$.
It is clear that the normalized Steklov eigenvalues depend only on $f_0(\theta)=f(0,\theta)$ and $f_1(\theta)=f(T,\theta)$.
So, we write them as $\tilde \sigma_k(f_0,f_1,T)$ for $k=0,1,2,\cdots$.
Let $\Gamma_0=\{0\}\times\sph^1$ and $\Gamma_1=\{T\}\times\sph^1$. Without loss of generality, we can assume that
$$\alpha:=\frac{L(\Gamma_0)}{L(\Gamma_1)}=\frac{\int_0^{2\pi} f_0(\theta)d\theta}{\int_0^{2\pi}f_1(\theta)d\theta}\geq 1.$$
Let $\beta=\frac{4\alpha}{(1+\alpha)^2}$ as before.

We want to compare $\ts_k(f_0,f_1,T)$ and $\ts_k(\beta,T)$. In general, it is not true that $\ts_k(f_0,f_1,T)\le \ts_k(\beta,T)$. A counter example will be given in next section. In this section, we prove that for a large class of $f$, $\ts_k(f_0,f_1,T)\le \ts_k(\beta,T)$ is true.
\begin{thm}\label{thm-eigen-comp-1}
When $T\geq T_{1,0}(\alpha)$, then
\begin{equation}
\tilde\sigma_1(f_0,f_1,T)\leq \tilde\sigma_1(\beta,T)=\tilde\mu_0(\beta,T)=\frac{8\pi}{T\beta}
\end{equation}
with equality holds if and only if $f_0$ and $f_1$ are constant functions.
\end{thm}
\begin{proof}
Let $g$ be the metric $f(t)^2(dt^2+d\theta^2)$ with $t\in [0,T]$,
$$2\pi f(0)=\int_0^{2\pi}f_0(\theta)d\theta$$
and
$$2\pi f(T)=\int_0^{2\pi}f_1(\theta)d\theta.$$
Note that the linear function $z_0(t)$ above is a first eigenfunction of $g$. So
\begin{equation}
\begin{split}
0=&\int_{\partial \Sigma}z_0dS_{g_0}\\
=&2\pi f(0)z_0(0)+2\pi f(T)z_0(T)\\
=&\int_{\Gamma_0}z_0(0)f_0(\theta)d\theta+\int_{\Gamma_1}z_0(T)f_1(\theta)d\theta\\
=&\int_{\partial\Sigma}z_0dS_{g}.
\end{split}
\end{equation}
Moreover, since $g$ and $\tilde g$ are conformal,
\begin{equation}
\begin{split}
&\int_\Sigma \|\nabla^{\tilde g}z_0\|^2dV_{\tilde g}\\
=&\int_\Sigma\|\nabla^{ g}z_0\|^2dV_{g}\\
=&\sigma_1(g)\int_{\partial\Sigma}z_0^2dS_{g}\\
=&\sigma_1(g)\int_{\partial\Sigma}z_0^2dS_{g}\\
=&\sigma_1(g)\int_{\partial\Sigma} z_0^2dS_{\tilde g}.\\
\end{split}
\end{equation}
This give us the inequality.

When equality holds, we know that $w_0$ is a first eigenfunction of $\tilde g$. Then
\begin{equation}
-\frac{1}{f_0(\theta)}z_0'(0)=\frac{\partial z_0}{\partial\nu}=\sigma_1z_0(0).
\end{equation}
Hence $f_0(\theta)=-\frac{z'_0(0)}{\sigma_1z_0(0)}$ is a constant. Similarly, this is true for $f_1$.
\end{proof}
\begin{rem}
There is a similar result in \cite{FraserSchoen-2011}.
\end{rem}
When $T<T_{1,0}(\alpha)$, we have the following partial comparison of the first Steklov eigenvalues.
\begin{thm}\label{thm-eigen-comp-2}
When $T<T_{1,0}(\alpha)$, suppose that the Fourier series of $f_0$ and $f_1$ are
\begin{equation}
f_0(\theta)=c(\alpha+\alpha_1\cos\theta+\alpha_2\sin\theta+\alpha_3\cos(2\theta)+\alpha_4\sin(2\theta)+\cdots)
\end{equation}
and
\begin{equation}
f_1(\theta)=c(1+\beta_1\cos\theta+\beta_2\sin\theta+\beta_3\cos(2\theta)+\beta_4\sin(2\theta)+\cdots).
\end{equation}
Let
\begin{equation}
A=\left(\begin{array}{ccc}-2(\alpha+1)&-(a\alpha_1+b\beta_1)&-(a\alpha_2+b\beta_2)\\
-(a\alpha_1+b\beta_1)&-\frac{1}{2}(a^2\alpha_3+b^2\beta_3)&-\frac{1}{2}(a^2\alpha_4+b^2\beta_4)\\
-(a\alpha_2+b\beta_2)&-\frac{1}{2}(a^2\alpha_4+b^2\beta_4)&\frac{1}{2}(a^2\alpha_3+b^2\beta_3)
\end{array}\right)
\end{equation}
where $a=\cosh(\tau_n)$ and $b=\cosh(T-\tau_n)$. Then, if $A$ has two nonpositive eigenvalues,
\begin{equation*}
\tilde\sigma_1(f_0,f_1,T)\leq \tilde\sigma_1(\beta,T)=\tilde\lambda_1(\beta,T),
\end{equation*}
with equality holds if and only if $f_0$ and $f_1$ are both constant functions.
\end{thm}
\begin{proof}
Without loss of generality, we can assume that the Fourier series of $f_0$ and $f_1$ are
\begin{equation}
f_0(\theta)=\alpha+\alpha_1\cos\theta+\alpha_2\sin\theta+\alpha_3\cos(2\theta)+\alpha_4\sin(2\theta)+\cdots
\end{equation}
and
\begin{equation}
f_1(\theta)=1+\beta_1\cos\theta+\beta_2\sin\theta+\beta_3\cos(2\theta)+\beta_4\sin(2\theta)+\cdots.
\end{equation}
Let $g=f^2(t)(dt^2+d\theta^2)$ with $f(0)=\alpha$ and $f(T)=1$ and
\begin{equation*}
Q(u,v)=\vv<L_{\tilde g}u,v>_{\tilde g}-\lambda_1\vv<u,v>_{\tilde g}=\vv<L_{ g}u,v>_{g}-\lambda_1\vv<u,v>_{\tilde g}
\end{equation*}
where
$$\vv<u,v>_g=\int_{\partial\Sigma}uvdS_g.$$
This is a quadratic form on the function space of $\partial\Sigma$.
By direct computation, the matrix of $Q$ on $\mbox{span}\{x_0,x_1,y_1\}$ with respect to the basis $x_0,x_1,y_1$ is $\lambda_1\pi A$. Because $A$ has two nonpositive eigenvalues, there is a nonzero function $u\in \mbox{span}\{x_0,x_1,y_1\}$, such that $\vv<u,x_0>_{\tilde g}=0$ and $Q(u,u)\leq 0$. Hence,
\begin{equation}
\sigma_1(f_0,f_1,T)\leq \lambda_1(\beta,T)=\sigma_1(\beta,T),
\end{equation}
and
\begin{equation}
\tilde \sigma_1(f_0,f_1,T)\leq \lambda_1(\beta,T)=\tilde\sigma_1(\beta,T).
\end{equation}
When equality holds, the same argument as in Theorem \ref{thm-eigen-comp-1} give us the conclusion.
\end{proof}
For other Steklov eigenvalues, we have the following partial comparison of eigenvalues.
\begin{thm}
Let $$E_{2k}=\mbox{span}\{\cos\theta,\sin\theta,\cos(2\theta),\sin(2\theta),\cdots,\cos(2k\theta),\sin(2k\theta)\}.$$
If $f_0,f_1\in E_{2k}^{\bot}$, then
\begin{equation}
\tilde\sigma_{2k-1}(f_0,f_1,T)\leq \tilde\sigma_{2k-1}(\beta,T)
\end{equation}
and
\begin{equation}
\tilde\sigma_{2k}(f_0,f_1,T)\leq \tilde\sigma_{2k}(\beta,T)
\end{equation}
with equality holds if and only if $f_0$ and $f_1$ are constants. Here $E_{2k}^{\bot}$ means the orthogonal complement of $E_{2k}$ in $L^2([0,2\pi])$.
\end{thm}
\begin{proof}
Let $g$ be the metric $f(t)^2(dt^2+d\theta^2)$ with $t\in [0,T]$,
$$2\pi f(0)=\int_0^{2\pi}f_0(\theta)d\theta$$
and
$$2\pi f(T)=\int_0^{2\pi}f_1(\theta)d\theta.$$
Let $\phi_0,\phi_1,\cdots,\phi_{2k}$ be the first $2k+1$ eigen-functions of $g$. Then, by the computation in the first section and the eigen-functions listed in the second section, we know that
\begin{equation}
\phi_i(0,\theta),\phi_i(T,\theta)\in \mbox{span}\{1,\cos\theta,\sin\theta,\cos(2\theta),\sin(2\theta),\cdots,\cos(k\theta),\sin(k\theta)\}
\end{equation}
for $i=0,1,2,\cdots,2k$.
Hence
\begin{equation}\label{eqn-i-j}
\vv<\phi_i,\phi_j>_{g}=\vv<\phi_i,\phi_j>_{\tilde g}
\end{equation}
for all $i,j=0,1,2,\cdots,2k$ since $f_0,f_1\in E_{2k}^\bot$.

Let $\psi_0,\psi_1,\cdots,\psi_{2k-1}$ be the first $2k$ eigen-functions of $\tilde g$. It is clear that there are constants $c_0,c_1,\cdots,c_{2k-2}$ that are not all zero, such that $u=c_0\phi_0+\cdots+c_{2k-2}\phi_{2k-1}$ is orthogonal to $\psi_0,\psi_1,\cdots,\psi_{2k-2}$ with respect to $\tilde g$. Moreover,
\begin{equation}
\begin{split}
&\vv<L_{\tilde g}u,u>_{\tilde g}\\
=&\vv<L_{g}u,u>_{g}\\
=&\vv<\sum_{i=0}^{2k-1}\sigma_i(g)c_i\phi_i,\sum_{i=0}^{2k-1}c_i\phi_i>_{g}\\
=&\sum_{i=0}^{2k-1}\sigma_i(g)c_i^2\vv<\phi_i,\phi_i>_{g}\\
\leq&\sigma_{2k-1}(g)\sum_{i=0}^{2k-1}c_i^2\vv<\phi_i,\phi_i>_{g}\\
=&\sigma_{2k-1}(g)\sum_{i=0}^{2k-1}c_i^2\vv<\phi_i,\phi_i>_{\tilde g}\\
=&\sigma_{2k-1}(g)\vv<u,u>_{\tilde g}\\
\end{split}
\end{equation}
by \eqref{eqn-i-j}. Hence,
$$\sigma_{2k-1}(f_0,f_1,T)\leq \sigma_{2k-1}(g)=\sigma_{2k-1}(\beta,T).$$
For the equality case, the argument is similar as in Theorem \ref{thm-eigen-comp-1}.

Similarly, we have $$\sigma_{2k}(f_0,f_1,T)\leq \sigma_{2k}(g)=\sigma_{2k}(\beta,T)$$
with equality holds if and only if $f_0$ and $f_1$ are constants.
\end{proof}
\section{A counter example}
Let $\Sigma=[0,T]\times \sph^1$ with conformal metric $\mathbf{g}=f^2(t,\theta)(dt^2+d\theta^2)$. Consider the rotational symmetric metric  $\mathbf{g}_0=h(t)^2(dt^2+d\theta^2)$ such that

$$
h(0)=\frac1{2\pi}\int_0^{2\pi}f(0,\theta)d\theta, \ h(T)=\frac1{2\pi}\int_0^{2\pi}f(T,\theta)d\theta.
$$
In this section, we want to construct an example, such that $\sigma_1(\mathbf{g})> \sigma_1(\mathbf{g}_0)$. Hence $\ts_1(\mathbf{g})> \ts_1(\mathbf{g}_0)$.

Let $\mbg =(1+\frac{1}{2}\cos\theta+\frac{1}{8}\cos(2\theta))(dt^2+d\theta^2)$ and $\mbg_0=dt^2+d\theta^2$ on the cylinder $\Sigma=[0,T]\times \sph^1$. Note that
$L_{\mbg}$ and $L_{\mbg_0}$ are operators on $C^\infty(\partial \Sigma)=C^\infty(\sph^1)\times C^\infty(\sph^1)$. Then, by the eigenvalues and eigenfunctions listed in the second section, we know that all the normalized eigenvalues and eigen-spaces of $L_{\mbg_0}$ are as follows:
\begin{enumerate}
\item $\tilde \lambda_0=0$ with eigen-space generated by $\mathbf x_0=(1,1)$;
\item $\tilde\lambda_n=4n\pi\tanh(\frac{nT}{2})$ with eigen-space generated by $\mathbf x_n=(\cos(n\theta),\cos(n\theta))$ and $\mathbf y_n=(\sin(n\theta),\sin(n\theta))$ for $n=1,2,\cdots$;
\item $\tilde \mu_0=\frac{8\pi}{T}$ with eigen-space generated by $\mathbf z_0=(1,-1)$;
\item  $\tilde \mu_n=4n\pi\coth(\frac{nT}{2})$ with eigen-space generated by $\mathbf z_n=(\cos(n\theta),-\cos(n\theta))$ and $\mathbf w_n=(\sin(n\theta),-\sin(n\theta))$ for $n=1,2,\cdots$.
\end{enumerate}

\begin{thm}\label{thm-example}
When $T<T_{1,0}(1)$ is small enough, then
 $\sigma_1(f,f,T)>\lambda_1(1,T)=\sigma_1(1,T)$.
\end{thm}

 Let $\mathbf{u}=(u(0),u(T))$ be an eigenfunction of $\sigma_1(\mathbf{g})$. Then
 $$L_{\mathbf{g}}(\bu)=\sigma_1(\mbg) \bu, $$
   and

 \begin{equation}
\begin{split}
\mathbf{u}=&\sum_{n=0}^\infty\lf( a_n\cos(n\theta)+b_n\sin(n\theta)\ri) \mathbf{x_0}+\sum_{n=0}^\infty\lf( c_n\cos(n\theta)+d_n\sin(n\theta)\ri) \mathbf{z_0}
\end{split}
\end{equation}
where $\bx_0=(1,1)$ and $\mathbf{z_0}=(1,-1)$.

Since $\mathbf{v}=(u(T),u(0))$ is also an eigenfunction for $\sigma_1$. By considering $\mathbf{u}+\mathbf{v}$ and  $\mathbf{u}-\mathbf{v}$, we may assume that either $u(0)=u(T)$ or $u(0)=-u(T)$. Hence we may assume either,
\be\label{eq-even}
\mathbf{u}= \sum_{n=0}^\infty\lf( a_n\cos(n\theta)+b_n\sin(n\theta)\ri) \mathbf{x_0},
\ee
or
\be\label{eq-odd}
\mathbf{u}=\sum_{n=0}^\infty\lf( c_n\cos(n\theta)+d_n\sin(n\theta)\ri) \mathbf{z_0}.
\ee
\begin{lem}\label{lem-1}

$$\lim_{T\to\infty}\frac{\lambda_k(1,T)}{\lambda_l(1,T)}=\frac{k^2}{l^2}.$$
Moreover, for $k\ge l$,
$$
\frac{k}l\le \frac{\lambda_k(1,T)}{\lambda_l(1,T)}\le \frac{k^2}{l^2}.
$$
\end{lem}
\begin{proof}
Note that $\tilde\lambda_k(1,T)=4k\pi\tanh(kT/2)$, hence
\begin{equation}
\frac{\lambda_k(1,T)}{\lambda_l(1,T)}=\frac{k\tanh(kT/2)}{l\tanh(lT/2)}\geq \frac{k}{l}.
\end{equation}
Moreover, the function $f(t)=a\tanh(at)/\tanh(t)$ is decreasing on $t>0$ for $a\geq 1$. Hence
\begin{equation}
\frac{k\tanh(kT/2)}{l\tanh(lT/2)}\leq\lim_{T\to0^+}\frac{k\tanh(kT/2)}{l\tanh(lT/2)}=\frac{k^2}{l^2}.
\end{equation}
\end{proof}
\begin{lem}\label{lma-odd} If $T>0$ is small enough, and if $\bu$ is of the form \eqref{eq-even}, then $\sigma_1=\sigma_1(\mathbf{g})>\sigma_1(\mbg_0)$.
\end{lem}
\begin{proof} If $T>0$ is small enough, then $\sigma_1(\mbg_0)=\lambda_1(1,T) $.   Then
\bee
\begin{split}
\sigma_1\la \bu,\bu\ra_{\mbg}=&  \la \bu,L_{\mbg}\bu\ra_{\mbg}\\
=&  \la \bu,L_{\mbg_0}\bu\ra_{\mbg_0}\\
\ge & \mu_0(1,T)\la\bu,\bu\ra_{\mbg_0}.
\end{split}
\eee
On the other hand,
\bee
\begin{split}
\lambda_1(1,T)\la \bu,\bu\ra_{\mbg}\le &||f||_{\infty}\lambda_1(1,T)\la \bu,\bu\ra_{\mbg_0}.
\end{split}
\eee
Now $\mu_0(1,T)\to\infty$ as $T\to0$ and $\lambda_1(1,T)\to0$ as $T\to0$. From these the lemma follows.
\end{proof}

\begin{lem}\label{lma-even} If $T>0$  is small enough, and if $\bu$ is of the form \eqref{eq-odd}, then $\sigma_1=\sigma_1(\mathbf{g})>\sigma_1(\mbg_0)$.
\end{lem}
\begin{proof}
$\bu=\bu_1+\bu_2$, where
$$
\bu_1=\sum_{n=0}^\infty a_n\cos(n\theta)\mathbf x_0,\ \bu_2=\sum_{n=1}^\infty b_n\sin(n\theta)\mathbf x_0.
$$
As in the proof of the previous lemma, for any $\e>0$, there is $T_0>0$ such that if $0<T<T_0$
\be\label{eq-even-1}
\begin{split}
\frac{\sigma_1}{2\pi\lambda_1(1,T)}\la \bu,\bu\ra_{\mbg}=&
     \sum_{n=1}^\infty \frac{\lambda_n(1,T)}{\lambda_1(1,T)}(a_n^2+b_n^2)\\ \ge& (a_1^2+b_1^2)+(4-\e)(a_2^2+b_2^2)+(9-\e)\sum_{n=3}^\infty b_n^2,
\end{split}
\ee
where we have used Lemma \ref{lem-1}. Now
\be\label{eq-even-2}
\begin{split}
\frac1{2\pi}\la \bu,\bu\ra_{\mbg}=&\frac1\pi\int_0^{2\pi}f(\theta)\lf(\sum_{n=0}^\infty a_n\cos(n\theta)+b_n\sin(n\theta)\ri)^2d\theta\\
=&\frac1\pi\int_0^{2\pi}f(\theta)\lf(\sum_{n=0}^\infty a_n\cos(n\theta)\ri)^2d\theta+\frac1\pi\int_0^{2\pi}f(\theta)\lf(\sum_{n=1}^\infty b_n\sin(n\theta)\ri)^2\\
=I+II
\end{split}
\ee
where we have used the fact that $f(\theta)\bu_1$ is a series in $\cos(n\theta)$.
Now
\be\label{eq-even-3}
\begin{split}
II=&\sum_{n=1}^\infty b_n^2+\frac{1}{ \pi}\int_0^{2\pi}\lf(\frac12\cos\theta+\frac18\cos(2\theta)\ri)\lf(\sum_{n=1}^\infty b_n\sin(n\theta)\ri)^2d\theta \\
\le&\lf(1+\frac12+\frac18\ri)\sum_{n=3}^\infty b_n^2\\
&+b_1^2+b_2^2+\frac{1}{ \pi}\int_0^{2\pi}\lf(\frac12\cos\theta+\frac18\cos(2\theta)\ri)
\lf(b_1\sin\theta+b_2\sin(2\theta)\ri)
\cdot\\
&\cdot\lf(b_1\sin\theta+b_2\sin(2\theta)+2\sum_{n=3}^\infty
 b_n\sin(n\theta)\ri) d\theta\\
 =&\frac{13}8\sum_{n=3}^\infty b_n^2+b_1^2+b_2^2\\
&+\frac{1}{ \pi}\int_0^{2\pi}\lf( \lf(\frac{b_2}4-\frac{b_1}{16}\ri)\sin\theta+\frac{b_1}4\sin(2\theta)+\lf(\frac{b_2}4+\frac{b_1}4\ri)\sin(3\theta)
+\frac{b_2}{16}\sin(4\theta)\ri)
\cdot\\
&\cdot\lf(b_1\sin\theta+b_2\sin(2\theta)+2\sum_{n=3}^\infty
 b_n\sin(n\theta)\ri) d\theta\\
 =&\frac{13}8\sum_{n=3}^\infty b_n^2+b_1^2+b_2^2+b_1\lf(\frac{b_2}4-\frac{b_1}{16}\ri)+\frac{b_1b_2}4+\frac{ b_3 (b_1+b_2)}2+\frac{b_2b_4}8\\
 \le&\frac{13}8\sum_{n=3}^\infty b_n^2+b_1^2+\lf(3+\frac 5{16}\ri)b_2^2+2b_2^2+2b_3^2
 +\frac94b_3^2+\frac1{16}b_4^2.
\end{split}
\ee
To estimate $I$, note that $\la\bu,1\ra_{\mbg}=0$. So
\be\label{eq-even-4}
a_0+\frac14a_1+\frac1{16}a_2=0.
\ee
\be\label{eq-even-5}
\begin{split}
I=&2a_0^2+\sum_{n=1}^\infty a_n^2+ \frac{1}{ \pi}\int_0^{2\pi}\lf(\frac12\cos\theta+\frac18\cos(2\theta)\ri)
\lf(\sum_{n=0}^\infty a_n\cos(n\theta)\ri)^2d\theta\\
\le&2a_0^2+a_1^2+a_2^2+\frac{13}8\sum_{n=3}^\infty a_n^2\\
& +\frac{1}{ \pi}\int_0^{2\pi}\lf(\frac12\cos\theta+\frac18\cos(2\theta)\ri)
\lf(a_0+a_1\cos\theta+a_2\cos(2\theta)\ri)
\cdot\\
&\cdot\lf(a_0+a_1\cos\theta+a_2\cos(2\theta)+2\sum_{n=3}^\infty
 a_n\cos(n\theta)\ri) d\theta\\
 =&2a_0^2+a_1^2+a_2^2+\frac{13}8\sum_{n=3}^\infty a_n^2\\
& +\frac{1}{ \pi}\int_0^{2\pi} \bigg( \frac14a_1+\frac1{16}a_2+\lf(\frac{a_2}4+\frac12a_0+\frac1{16}a_1\ri)\cos\theta+
\lf(\frac4a_1+\frac18a_0\ri)\cos(2\theta)\\&+\lf(\frac14a_2+\frac1{16}a_1\ri)\cos(3\theta) +\frac1{16}a_2\cos(4\theta)\bigg)\cdot
 \\
 &\cdot\lf(a_0+a_1\cos\theta+a_2\cos(2\theta)+2\sum_{n=3}^\infty
 a_n\cos(n\theta)\ri) d\theta\\
 =&2a_0^2+a_1^2+a_2^2+\frac{13}8\sum_{n=3}^\infty a_n^2+ a_0a_1+\frac14a_0a_2+\frac12a_1a_2+\frac1{16}a_1a_3+\frac1{16}a_2a_4+\frac14a_2a_3\\
 =&-2a_0^2+a_1^2+a_2^2+\frac{13}8\sum_{n=3}^\infty a_n^2+\frac12a_1a_2+\frac1{16}a_1a_3+\frac1{16}a_2a_4+\frac14a_2a_3\\
 \le&\frac{13}8\sum_{n=3}^\infty a_n^2+a_1^2+a_2^2-\frac18 a_1^2+\frac12 |a_1||a_2|+\frac1{16}a_1a_3+\frac1{16}a_2a_4+\frac14a_2a_3\\
  \le&\frac{13}8\sum_{n=3}^\infty a_n^2+a_1^2+a_2^2 -\frac18 a_1^2+\frac12 |a_1||a_2|\\
  &+7a_3^2+\frac{1}{7\times 32^2}a_1^2+7a_4^2+\frac{1}{7\times 32^2}a_2^2+\frac18a_2^2+\frac18a_3^2\\
  &\le a_1^2+\lf(1+\frac14\ri)a_2^2+\lf(7+\frac{14}8\ri)a_3^2+\lf(7+\frac{13}8\ri)a_4^2+\frac{13}8
  \sum_{n=5}^\infty a_n^2
\end{split}
\ee
where we have used \eqref{eq-even-4}
By \eqref{eq-even-1}--\eqref{eq-even-3} and \eqref{eq-even-5}, one can conclude that the lemma is true.
\end{proof}

\begin{proof}[Proof of Theorem \ref{thm-example}] The theorem follows from Lemmas \ref{lma-odd} and \ref{lma-even}.
\end{proof}


\begin{thebibliography}{10}
\bibitem{Calderon1980}Calder\'on, Alberto-P. {\sl On an inverse boundary value problem. Seminar on Numerical Analysis and its Applications to Continuum Physics (Rio de Janeiro, 1980)}, pp. 65--73, Soc. Brasil. Mat., Rio de Janeiro, 1980.
\bibitem{Cheng1975}Cheng, Shiu Yuen.{\sl Eigenvalue comparison theorems and its geometric applications.} Math. Z. 143 (1975), no. 3, 289--297.
\bibitem{CEG2001}Colbois, Bruno; El Soufi, Ahmad; Girouard, Alexandre. {\sl Isoperimetric control of the Steklov spectrum.} J. Funct. Anal. 261 (2011), no. 5, 1384--1399.
\bibitem{Escobar2000}Escobar, Jos\'e F. {\sl A comparison theorem for the first non-zero Steklov eigenvalue.} J. Funct. Anal. 178 (2000), no. 1, 143--155.
\bibitem{FraserSchoen-2011} Fraser, A.; Schoen, R., {\sl
The first Steklov eigenvalue, conformal geometry, and minimal surfaces,} Adv. Math. \textbf{226} (2011), no. 5, 4011--4030.

 \bibitem{FraserSchoen-2012} Fraser, A.; Schoen, R., {\sl Sharp eigenvalue bounds and minimal surfaces in the ball}, preprint, arXiv:1209.3789
 \bibitem{FraserSchoen-2013} Fraser, A.; Schoen, R., {\sl Minimal surfaces and eigenvalue problems}, preprint, arXiv:1304.0851.
 \bibitem{GP2010}Girouard, A.; Polterovich, I. {\sl On the Hersch-Payne-Schiffer estimates for the eigenvalues of the Steklov problem.} (Russian) Funktsional. Anal. i Prilozhen. 44 (2010), no. 2, 33--47; translation in Funct. Anal. Appl. 44 (2010), no. 2, 10--117.
 \bibitem{LU2001}Lassas, Matti; Uhlmann, Gunther. {\sl On determining a Riemannian manifold from the Dirichlet-to-Neumann map.} Ann. Sci. ¨¦cole Norm. Sup. (4) 34 (2001), no. 5, 771--787.
 \bibitem{LTU2003}Lassas, Matti; Taylor, Michael; Uhlmann, Gunther. {\sl The Dirichlet-to-Neumann map for complete Riemannian manifolds with boundary.} Comm. Anal. Geom. 11 (2003), no. 2, 207--221.
 \bibitem{LY1982} Li, Peter; Yau, Shing Tung. {\sl A new conformal invariant and its applications to the Willmore conjecture and the first eigenvalue of compact surfaces.} Invent. Math. 69 (1982), no. 2, 269--291.
\bibitem{Steklov}Stekloff, W.; {\sl Sur les probl\`emes fondamentaux de la physique math\'e matique.} (French) Ann. Sci. \'Ecole Norm. Sup. (3) 19 (1902), 191--259.
\bibitem{W}Weinstock, Robert. {\sl Inequalities for a classical eigenvalue problem.} J. Rational Mech. Anal. 3, (1954). 745¨C753.


\end{thebibliography}
\end{document}